\documentclass[11pt]{article}
\usepackage[margin=1.1in]{geometry}
\usepackage{amsmath,amsthm,amssymb,mathtools}
\usepackage{hyperref}
\usepackage{booktabs}

\theoremstyle{plain}
\newtheorem{theorem}{Theorem}[section]
\newtheorem{lemma}[theorem]{Lemma}
\newtheorem{proposition}[theorem]{Proposition}
\newtheorem{corollary}[theorem]{Corollary}
\newtheorem{conjecture}[theorem]{Conjecture}
\theoremstyle{definition}
\newtheorem{definition}[theorem]{Definition}
\newtheorem{example}[theorem]{Example}
\newtheorem{remark}[theorem]{Remark}

\newcommand{\sh}{\mathbin{\text{\rlap{$\sqcup$}$\sqcup$}}}
\newcommand{\NN}{\mathbb{N}}
\newcommand{\ZZ}{\mathbb{Z}}
\newcommand{\QQ}{\mathbb{Q}}

\title{Improved upper bound on the number of distinct k-decks for any k and alphabet size by counting the independent parameters}
\author{Arman Nilforoushan \and Farzad Parvaresh
\thanks{This work has been submitted to the IEEE for possible publication.
Copyright may be transferred without notice, after which this version may no
longer be accessible.}}
\date{}
\begin{document}
\maketitle
\begin{abstract}
Data stored in synthetic DNA is retrieved by shotgun sequencing, which returns
short subsequences rather than the stored word itself. A natural abstraction of
this readout is the $k$-deck of a word: the vector recording how often each word
of length $k$ occurs as a subsequence. Two stored words are distinguishable from
their readouts exactly when their $k$-decks differ, so the number $D_{q,k}(n)$ of
distinct $k$-decks of words of length $n$ over an alphabet of size $q$ measures
what a length-$k$ readout retains.

We analyse the degrees of freedom remaining in a $k$-deck once all shorter decks
are fixed. Within each class of words having prescribed letter multiplicities,
the length-$k$ entries are confined to an affine subspace whose dimension is
exactly the number of Lyndon words with the same multiplicities, which we give in
closed form as a M\"obius sum. Writing $L_q(j)$ for the number of Lyndon words of
length $j$ over an alphabet of size $q$, we deduce the improved upper bound
\[
  D_{q,k}(n)=O\!\left(n^{E_q(k)}\right),\qquad E_q(k)=\sum_{j=1}^{k}j\,L_q(j)-1 .
\]
In the case of a binary alphabet this bound satisfies
$D_{2,k}(n)=O\!\left(n^{4\cdot 2^{k-1}}\right)$.

We then prove matching lower bounds in the first two nontrivial cases:
$D_{q,2}(n)=\Theta\!\left(n^{q^2-1}\right)$ for every alphabet size $q$, and
$D_{2,3}(n)=\Theta(n^{9})$ for the binary alphabet. The latter confirms, for $q=2$ and $k=3$, our
conjecture that the upper bound has the correct degree for every $q$ and $k$.
\end{abstract}


\section{Introduction}

\subsection{DNA-based storage and its readout channel}

Synthetic DNA has emerged as a medium for archival data storage, offering very
high density and long-term stability without power \cite{church,goldman,grass,yazdi}.
A storage system built on it has three components: a synthesizer that writes the
encoded data onto short DNA strands, an unordered pool in which the strands are
held, and a sequencer that reads them back.

The read step is what makes the channel unusual. Current technologies cannot read
a long strand directly and reliably; instead, shotgun sequencing fragments the
strands and returns a large multiset of short reads, from which the stored word
must be reconstructed. The information that reaches the decoder is therefore not
the word itself but a collection of statistics about its short factors.

Kiah, Puleo and Milenkovic \cite{kiah} proposed profile vectors as the right
abstraction of this readout, and studied codes designed so that distinct
codewords have distinct profiles. A profile vector records how many times each
short word occurs inside the stored word. Taking the occurrences to be
subsequences rather than contiguous factors---the convention we adopt, following
\cite{chrisnata,manvel}---gives the $k$-deck.

\begin{definition}[$k$-deck]\label{def:deck}
Let $\mathcal A$ be an alphabet of size $q$ and let $X\in\mathcal A^n$. For
$u\in\mathcal A^k$ write $X_u$ for the number of occurrences of $u$ as a
subsequence of $X$, that is, the number of index tuples $i_1<\cdots<i_k$ with
$X_{i_1}\cdots X_{i_k}=u$. The \emph{$k$-deck} of $X$ is the vector
$\left(X_u\right)_{u\in\mathcal A^{k}}$.
\end{definition}

The operational meaning is immediate. If two stored words have the same $k$-deck,
then no decoder whose input consists only of length-$k$ subsequence statistics can
tell them apart. Consequently, writing
\[
  D_{q,k}(n):=\#\bigl\{\text{distinct $k$-decks of words in }\mathcal A^n\bigr\},
\]
the quantity $D_{q,k}(n)$ is exactly the number of messages such a channel can
convey without error, and $\log_q D_{q,k}(n)$ is its zero-error capacity in
symbols per block of length $n$. Determining the growth of $D_{q,k}(n)$ in $n$ is
thus the basic capacity question for the model. (The identification of
$D_{q,k}(n)$ with the largest codebook whose codewords are pairwise
distinguishable is the coded $k$-deck problem of \cite{chrisnata}; we use
``zero-error capacity'' as a reformulation of it, not as established terminology
for this channel.)

\begin{example}\label{ex:small}
Take $q=2$, $n=4$, $k=2$. The word $0011$ has $2$-deck
$(X_{00},X_{01},X_{10},X_{11})=(1,4,0,1)$, while $0101$ has $(1,3,1,1)$ and
$0110$ has $(1,2,2,1)$. Among the sixteen binary words of length four exactly one
pair collides, namely $0110$ and $1001$, which share the deck $(1,2,2,1)$; hence
$D_{2,2}(4)=15$. Already at this smallest scale the readout has merged two
distinct messages.
\end{example}

\subsection{Independent parameters: the idea}\label{sec:idea}

The $k$-deck has $q^k$ entries, but they are far from free. Some are determined
outright by the word length and the letter counts; others are tied together by
identities valid for every word. Any bound on $D_{q,k}(n)$ obtained by
enumerating decks must therefore begin by asking how many entries can genuinely
vary.

The following instance, for $q=2$ and $k=3$, shows the mechanism. Let $X$ be
binary of length $n$ with $w$ ones, and write $t=X_{01}$. Counting pairs of
positions shows
\begin{equation}\label{eq:level2}
  X_{00}=\binom{n-w}{2},\qquad X_{11}=\binom{w}{2},\qquad X_{01}+X_{10}=w(n-w),
\end{equation}
so that, with $n$ and $w$ held fixed, the entire $2$-deck is determined by the
single number $t$. At length three the same style of count gives, for example,
\begin{equation}\label{eq:level3}
  X_{001}+X_{010}+X_{100}=w\binom{n-w}{2},\qquad 2X_{001}+X_{010}=t(n-w-1),
\end{equation}
and three further relations for the entries of weight two. Solving, one finds
that once $n$, $w$ and $t$ are fixed, the eight entries of the $3$-deck are
determined by just two of them, for instance $X_{001}$ and $X_{011}$.

So at length $1$ there is one free parameter, at length $2$ one more, and at
length $3$ two more. The pattern $1,1,2$ can represent numerous sequences but
we'll show that these are the numbers of binary Lyndon words of lengths $1,2,3$ after
accounting for the fixed word length. Lyndon words are defined, and this count made 
precise, in Section~\ref{sec:lyndon}; Proposition~\ref{prop:dof} says that the coincidence
is not one, and holds for every $q$ and $k$.

The relations above are not accidents of small cases. They all are a result of single 
product rule for subsequence counts, stated as Lemma~\ref{lem:product} in
Section~\ref{sec:product}. That rule is the infiltration product of Chen, Fox and
Lyndon \cite{cfl}, and its leading part is the shuffle product. Once this is
recognised, the classical structure theory of the shuffle
algebra---Chen--Fox--Lyndon factorization (Lemma~\ref{lem:cfl}) together with
Radford's theorem (Lemma~\ref{lem:radford})---determines the number of free
parameters.

\subsection{Prior work}

Reconstruction of a sequence from its $k$-deck was studied by Manvel, Meyerowitz,
Schwenk, Smith and Stockmeyer \cite{manvel}, who asked for the least $k$ such
that all binary words of length $n$ have distinct $k$-decks; the best known bounds
on that threshold are $\exp\bigl(c\sqrt{\log n}\bigr)$ from below, due to Dud\'ik
and Schulman \cite{dudik}, and $(2\sqrt{\ln 2}+\varepsilon)\sqrt n$ from above,
due to Foster and Krasikov \cite{foster}.

The counting problem addressed here---how many distinct $k$-decks occur---was
taken up by Chrisnata, Kiah, Rao, Vardy, Yaakobi and Yao
\cite{chrisnata,chrisnata-j}, who computed $D_{2,k}(n)$ exactly for $k\le 6$ and
$n\le 30$ by a trellis construction and proved
\begin{equation}\label{eq:old}
  D_{2,k}(n)=O\!\left(n^{(k-1)2^{k-1}+1}\right),
\end{equation}
improving an earlier bound of Rigo and Salimov \cite{rigo}, together with the
lower bounds $D_{2,k}(n)\ge\binom{n-k}{k}=\Omega(n^{k})$ in general and
$D_{2,3}(n)=\Omega(n^{6})$. For general alphabets, Lejeune, Rigo and Rosenfeld
\cite{lejeune} proved $D_{q,2}(n)=\Theta(n^{q^2-1})$ and
$D_{q,k}(n)=O(n^{k^2q^k})$ for $k\ge3$; their $k=2$ result is recovered
independently, and by an elementary construction, in
Section~\ref{sec:k2} below.

\subsection{Contributions}

\begin{itemize}
\item[(C1)] Once all decks of length at most $k-1$ are fixed, the length-$k$
entries of a given content are confined to an affine subspace whose dimension is
exactly the number of Lyndon words of that content; we give that number in closed
form (Proposition~\ref{prop:dof} and Corollary~\ref{cor:closed}).
\item[(C2)] We deduce $D_{q,k}(n)=O\!\left(n^{E_q(k)}\right)$ for all $q$ and $k$,
where
\[
  E_q(k)=\sum_{j=1}^{k} j\,L_q(j)-1,\qquad
  L_q(j)=\frac1j\sum_{d\mid j}\mu(d)\,q^{j/d},
\]
and $L_q(j)$ is the number of Lyndon words of length $j$ over an alphabet of size
$q$. For instance $L_2(1),\dots,L_2(4)=2,1,2,3$, so $E_2(4)=2+2+6+12-1=21$. See
Theorem~\ref{thm:upper}. For $q=2$ compared to \eqref{eq:old} the two agree at $k=3$, where both give $9$, and the
gain first appears at $k=4$.

\item[(C3)] We prove $D_{q,2}(n)=\Theta\!\left(n^{q^2-1}\right)$ for every
$q\ge2$ (Theorem~\ref{thm:k2}), by an elementary segment construction. This
recovers a theorem of \cite{lejeune} without machine assistance.

\item[(C4)] We prove $D_{2,3}(n)=\Theta(n^{9})$ (Theorem~\ref{thm:main9}),
improving the $\Omega(n^6)$ of \cite{chrisnata} and closing the case $q=2$,
$k=3$.

\item[(C5)] We conjecture $D_{q,k}(n)=\Theta\!\left(n^{E_q(k)}\right)$ in
general; by (C3) and (C4) it holds for all $q$ when $k\le2$ and for $q=2$ when
$k\le3$.
\end{itemize}

\subsection{Organization}

Section~\ref{sec:prelim} fixes notation. Section~\ref{sec:product} establishes the
product rule. Section~\ref{sec:lyndon} introduces Lyndon words and proves the
degrees-of-freedom count. Section~\ref{sec:upper} derives the upper bound.
Section~\ref{sec:k2} proves the matching lower bound for $k=2$ and every $q$.
Section~\ref{sec:lower3} is independent of Sections~\ref{sec:product}--\ref{sec:upper}
and proves the matching lower bound for $q=2$, $k=3$.

\section{Preliminaries}\label{sec:prelim}

Throughout, $\mathcal A=\{a_1<a_2<\cdots<a_q\}$ is a totally ordered alphabet,
and $\mathcal A^*$ denotes the free monoid on $\mathcal A$, i.e., the set of all
finite words over $\mathcal A$, including the empty word and $\mathcal A^+:=\mathcal
A^*\setminus\{\varepsilon\}$ denotes the set of nonempty words.
For $X,u\in\mathcal A^*$,
\[
  X_u:=\binom{X}{u}=\#\bigl\{i_1<\cdots<i_{|u|}\ :\ X_{i_1}\cdots X_{i_{|u|}}=u\bigr\}.
\]
We write $\mu$ for the M\"obius function of elementary number theory: $\mu(1)=1$,
$\mu(d)=(-1)^r$ if $d$ is a product of $r$ distinct primes, and $\mu(d)=0$ if $d$
is divisible by the square of a prime.

\begin{definition}[Content]\label{def:content}
The \emph{content} of $u$ is $c(u)=(m_1,\dots,m_q)\in\NN^q$, where $m_i$ counts
the occurrences of $a_i$ in $u$; thus $|c(u)|:=\sum_i m_i=|u|$. We write
$\binom{k}{m}$ for the multinomial coefficient $\binom{k}{m_1,\dots,m_q}$ and
$\gcd(m)$ for $\gcd(m_1,\dots,m_q)$.
\end{definition}

We work in the free associative algebra $\QQ\langle\mathcal A\rangle$, graded by
content; this refines the grading by length, and the content-$m$ component
$\QQ\langle\mathcal A\rangle_m$ is spanned by the $\binom{|m|}{m}$ words of that
content.

\begin{example}
Over $\mathcal A=\{A,C,G,T\}$ the word $u=ACAG$ has content $c(u)=(2,1,1,0)$ and
$|c(u)|=4=|u|$. Its content class consists of the $\binom{4}{2,1,1,0}=12$ words
containing two $A$'s, one $C$ and one $G$. Here $\gcd(m)=1$. Over the binary
alphabet, content is equivalent to the pair (length, weight).
\end{example}

\begin{definition}[Shuffle]\label{def:shuffle}
$u\sh v$ is the formal sum, with multiplicity, of all interleavings of $u$ and
$v$ that preserve the internal order of each; extend bilinearly. It is
homogeneous of content $c(u)+c(v)$, so preserves content and hence length.
\end{definition}

\begin{example}\label{ex:shuffle}
Interleaving $AC$ with $G$, keeping $A$ before $C$, gives
$AC\sh G=ACG+AGC+GAC$, one term for each of the $\binom31=3$ ways of choosing the
slot for $G$. Repeated letters produce multiplicities: labelling the two copies
in $01\sh01$ shows that four of the six interleavings spell $0011$ and two spell
$0101$, so
\[
  01\sh01=4\,(0011)+2\,(0101),
\]
the total multiplicity being $\binom42=6$. In both cases every term has the same
content as the sum of the contents of the factors; this homogeneity is what
confines the constraints of Section~\ref{sec:lyndon} to a single content class.
\end{example}

\section{A Product Rule for Subsequence Counts}\label{sec:product}

\subsection{The equations}

Section~\ref{sec:idea} showed that the entries of a $k$-deck satisfy identities
such as $X_{01}+X_{10}=X_0X_1$, and it is these identities that make most deck
entries redundant. Our goal in this section is to explain where all such
identities come from. They come from a single rule that computes any product
$X_uX_v$ in terms of individual deck entries.

The rule is not new: it is the infiltration product of Chen, Fox and Lyndon
\cite[Thm.~3.9]{cfl}, defined independently by Ochsenschl\"ager \cite{ochsen} in
exactly the language of binomial coefficients of words used here, with a textbook
treatment in Lothaire \cite[Ch.~6]{lothaire}. We give a self-contained
combinatorial derivation because the form of the coefficients, and not only the
existence of the product, is what Section~\ref{sec:upper} uses.

\subsection{The idea in one example}

Take $X=0101$. Direct counting gives
\[
  X_0=2,\quad X_1=2,\quad X_{01}=3,\quad X_{10}=1,\quad X_{001}=1,\quad X_{010}=1 .
\]
Let us compute the product $X_{01}\cdot X_0=3\cdot2=6$ in a second way.

The number $6$ counts pairs: one occurrence of $01$ together with one occurrence
of $0$. Writing occurrences as sets of positions (indexed $0,1,2,3$), the
occurrences of $01$ are $\{0,1\},\{0,3\},\{2,3\}$, and those of $0$ are
$\{0\},\{2\}$. The six pairs are listed in Table~\ref{tab:pairs}.

\begin{table}[h]
\centering
\begin{tabular}{cccc}
\toprule
occurrence of $01$ & occurrence of $0$ & union of positions & word spelled\\
\midrule
$\{0,1\}$ & $\{0\}$ & $\{0,1\}$   & $01$\\
$\{0,3\}$ & $\{0\}$ & $\{0,3\}$   & $01$\\
$\{2,3\}$ & $\{2\}$ & $\{2,3\}$   & $01$\\
$\{0,1\}$ & $\{2\}$ & $\{0,1,2\}$ & $010$\\
$\{0,3\}$ & $\{2\}$ & $\{0,2,3\}$ & $001$\\
$\{2,3\}$ & $\{0\}$ & $\{0,2,3\}$ & $001$\\
\bottomrule
\end{tabular}
\caption{The six pairs counted by $X_{01}\cdot X_0$, and the word each spells.}
\label{tab:pairs}
\end{table}

Two features drive everything that follows. First, taking the union of the two
position sets always produces an occurrence of some word---here $01$, $010$ or
$001$---so the product is a sum of deck entries. Second, the union is sometimes
shorter than $2+1=3$ positions. This happens exactly when the two occurrences
share a position, as in the first three rows; sharing is possible only when the
two letters involved are equal. Grouping the rows by the word they spell gives
\[
  X_{01}\cdot X_0=1\cdot X_{01}+1\cdot X_{010}+2\cdot X_{001}=3+1+2=6 .
\]
The coefficients $1,1,2$ do not depend on $X$; they count the possible ways of
overlaying $01$ and $0$.

\subsection{Overlays}

\begin{definition}[Overlay]\label{def:overlay}
Let $u$ and $v$ be words. An \emph{overlay} of $u$ and $v$ is a word $\gamma$
built by writing the letters of $u$ and of $v$ into a single row of slots so that
\begin{enumerate}
\item[(i)] the letters of $u$ appear in their original order, and likewise for
$v$; and
\item[(ii)] each slot holds either one letter of $u$, or one letter of $v$, or one
letter of each --- and in the last case the two letters must be equal, and the
slot holds that common letter.
\end{enumerate}
Two overlays are counted as different if the slots used by $u$ and by $v$ differ,
even when the resulting word is the same. We write $C_\gamma$ for the number of
overlays producing the word $\gamma$.
\end{definition}

An overlay has $|u|+|v|$ slots when no slot is shared, and one fewer for each
shared slot; thus $\max(|u|,|v|)\le|\gamma|\le|u|+|v|$.

\begin{example}\label{ex:overlay}
For $u=01$ and $v=0$ there are exactly four overlays. Writing $0_u,1_u$ for the
letters of $u$ and $0_v$ for the letter of $v$:
\[
\begin{array}{llll}
\text{(a)} & \{0_u,0_v\}\ \ 1_u & & \gamma=01\\
\text{(b)} & 0_v\ \ 0_u\ \ 1_u  & & \gamma=001\\
\text{(c)} & 0_u\ \ 0_v\ \ 1_u  & & \gamma=001\\
\text{(d)} & 0_u\ \ 1_u\ \ 0_v  & & \gamma=010
\end{array}
\]
Overlay (a) shares its first slot, which is legitimate because both letters there
are $0$. Overlays (b) and (c) spell the same word $001$ but place $v$ in
different slots, so they are counted separately; this is why $C_{001}=2$,
matching the two rows of Table~\ref{tab:pairs}. Also $C_{01}=1$ and $C_{010}=1$.
\end{example}

\subsection{The product rule}

\begin{lemma}[Product rule]\label{lem:product}
For all words $u,v$ and every word $X$,
\[
  X_u\,X_v=\sum_{\gamma}C_\gamma\,X_\gamma ,
\]
the sum running over all words $\gamma$ that arise as an overlay of $u$ and $v$,
with $C_\gamma$ as in Definition~\ref{def:overlay}.
\end{lemma}

\begin{proof}
Both sides count the same finite set,
\[
  P=\bigl\{(\mathbf i,\mathbf j)\ :\ \mathbf i\text{ an occurrence of }u\text{ in }X,\
  \mathbf j\text{ an occurrence of }v\text{ in }X\bigr\} .
\]
The left-hand side counts $P$: there are $X_u$ choices of $\mathbf i$ and,
independently, $X_v$ choices of $\mathbf j$.

For the right-hand side, given a pair let $K=\mathbf i\cup\mathbf j$, listed in
increasing order, a position lying in both being listed once. The letter of $u$
and the letter of $v$ sitting at such a shared position are both equal to the
letter of $X$ there, hence equal to each other, so condition (ii) of
Definition~\ref{def:overlay} is met. Recording, for each element of $K$, whether
it came from $\mathbf i$ only, from $\mathbf j$ only, or from both, therefore
yields an overlay $\pi$ of $u$ and $v$; let $\gamma_\pi$ be the word it spells.
Since the letter of $X$ at each position of $K$ is exactly the letter the overlay
places there, $K$ is an occurrence of $\gamma_\pi$ in $X$. So the pair gives rise
to a pair $(\pi,K)$ with $\pi$ an overlay and $K$ an occurrence of $\gamma_\pi$.

Conversely, given an overlay $\pi$ and an occurrence $K$ of $\gamma_\pi$, the
overlay says which slots belong to $u$ and which to $v$; reading those slots off
inside $K$ returns a pair $(\mathbf i,\mathbf j)$, and it is the only pair that
could have produced $(\pi,K)$. The two constructions are mutually inverse.

Consequently $|P|=\sum_\pi X_{\gamma_\pi}$, and collecting the overlays $\pi$
according to the word $\gamma_\pi$ they spell---there are $C_\gamma$ of them for
each $\gamma$---gives $|P|=\sum_\gamma C_\gamma X_\gamma$.
\end{proof}

\begin{remark}
Nothing in the proof used the alphabet size, so the rule holds over any alphabet.
The only place the alphabet enters is condition (ii): two letters may share a
slot only if they are equal.
\end{remark}

\begin{lemma}[Leading terms]\label{lem:leading}
In Lemma~\ref{lem:product}, the terms with $|\gamma|=|u|+|v|$ are precisely those
in which no slot is shared. These are the terms of the shuffle product $u\sh v$,
with the same coefficients. All remaining terms satisfy $|\gamma|<|u|+|v|$.
\end{lemma}

\begin{proof}
An overlay with $s$ shared slots produces a word of length $|u|+|v|-s$, so
$|\gamma|=|u|+|v|$ exactly when $s=0$. An overlay with no shared slot is
precisely an interleaving of the two words respecting the order of each, which is
what the shuffle product enumerates.
\end{proof}

\begin{example}
Continuing Example~\ref{ex:overlay}: of the four overlays of $01$ and $0$, the
three with no shared slot are (b), (c), (d), giving $01\sh0=2\,(001)+1\,(010)$;
the remaining overlay (a) shares a slot and contributes the shorter word $01$.
\end{example}

\begin{remark}\label{rem:caution}
It is tempting to think that a shared slot must reduce the number of $1$s as well
as the length. It need not: sharing two $0$s shortens the word but leaves the
number of $1$s unchanged. For instance
\[
  X_{011}X_0=2X_{0011}+X_{0101}+X_{0110}+X_{011},
\]
where the short term $X_{011}$ has the same number of $1$s as the long ones. This
is why, in Section~\ref{sec:upper}, we rely on Lemma~\ref{lem:leading}---which
speaks only about length---rather than on any claim about how sharing affects the
letter counts.
\end{remark}

\section{Lyndon Words and the Degrees of Freedom}\label{sec:lyndon}

\begin{definition}\label{def:lyndon}
A nonempty word $\ell$ is \emph{Lyndon} if $\ell<\sigma(\ell)$ lexicographically
for every nontrivial rotation $\sigma$. Put
\[
  I_q(k,m):=\#\{\ell\text{ Lyndon}:c(\ell)=m,\ |m|=k\},\qquad
  L_q(k):=\sum_{|m|=k}I_q(k,m).
\]
\end{definition}

\begin{example}[Binary]\label{ex:lyndon2}
Over $\{0,1\}$ with $0<1$, the Lyndon words of length at most four are
\[
  0,\ 1;\qquad 01;\qquad 001,\ 011;\qquad 0001,\ 0011,\ 0111,
\]
so $L_2(1),\dots,L_2(4)=2,1,2,3$. Of the eight words of length three, only $001$
and $011$ qualify: $010$ fails because the rotation $001$ is smaller, $101$ and
$110$ because they do not begin with the smallest letter they contain, and $000$
and $111$ because they are periodic. Comparing with Section~\ref{sec:idea}, the
counts $1,1,2$ of free parameters found there at lengths $1,2,3$ are
$L_2(1)-1$, $L_2(2)$ and $L_2(3)$.
\end{example}

\begin{example}[Quaternary]
Over $\{A,C,G,T\}$ the six Lyndon words of length two are the strictly increasing
pairs $AC,AG,AT,CG,CT,GT$, since $xy$ is Lyndon exactly when $x<y$; and there are
twenty of length three:
\[
\begin{array}{l}
AAC,\ AAG,\ AAT,\ ACC,\ ACG,\ ACT,\ AGC,\ AGG,\ AGT,\ ATC,\\
ATG,\ ATT,\ CCG,\ CCT,\ CGG,\ CGT,\ CTG,\ CTT,\ GGT,\ GTT.
\end{array}
\]
By contrast $CAG$ is not Lyndon, its rotation $AGC$ being smaller, and $ACAC$ is
not, being periodic. Exactly one rotation of an aperiodic word is Lyndon, which
is why these counts are those of aperiodic necklaces.
\end{example}

\begin{lemma}[Chen--Fox--Lyndon]\label{lem:cfl}
Every $w\in\mathcal A^+$ factors uniquely as $w=\ell_1\ell_2\cdots\ell_r$ with
$\ell_1\ge\cdots\ge\ell_r$ Lyndon. Consequently, for any content vector
$m\in\NN^q$, the map $w\mapsto\{\ell_1,\dots,\ell_r\}$ is a bijection
\[
  \{w:c(w)=m\}\ \longleftrightarrow\
  \Bigl\{\text{multisets of Lyndon words with }\textstyle\sum_i c(\ell_i)=m\Bigr\},
\]
so the right-hand set has cardinality $\binom{|m|}{m}$.
\end{lemma}

\begin{lemma}[Radford]\label{lem:radford}
Over a field of characteristic zero, $(\QQ\langle\mathcal A\rangle,\sh)$ is a
polynomial algebra on the Lyndon words \cite{radford}; see also
\cite[Ch.~6]{reutenauer}. Since $\sh$ adds contents, the content-$m$ component
has basis
\[
  B_m=\Bigl\{\ell_1\sh\cdots\sh\ell_r\ :\ \ell_1\ge\cdots\ge\ell_r\text{ Lyndon},\
  \textstyle\sum_i c(\ell_i)=m\Bigr\}.
\]
\end{lemma}

Lemma~\ref{lem:radford} supplies linear independence and Lemma~\ref{lem:cfl} the
matching cardinality $|B_m|=\binom{|m|}{m}=\dim\QQ\langle\mathcal A\rangle_m$.
That exactness is what makes the next lemma an equality rather than an inequality.

\begin{lemma}\label{lem:count-basis}
Write $k=|m|$. Exactly $I_q(k,m)$ elements of $B_m$ involve a factor of length
$k$, namely those with $r=1$; the remaining $\binom{k}{m}-I_q(k,m)$ have all
factors of length $<k$ and are linearly independent.
\end{lemma}

\begin{proof}
If $|\ell_i|=k$ for some $i$, then $\sum_j|\ell_j|=k$ forces $r=1$ and $\ell_1$
Lyndon of content $m$; there are $I_q(k,m)$ such. Independence is inherited from
$B_m$ being a basis.
\end{proof}

\subsection{Fixing the (k-i)-decks for i=1,2,...,k-1}

Suppose we already know every deck entry $X_\gamma$ for $|\gamma|\le k-1$, and we
now look at the entries of length exactly $k$. Because a shuffle never changes
the number of occurrences of any letter, these split into independent groups
according to their content $m$, and we treat one group at a time. The group of
content $m$ has $\binom{k}{m}$ entries, and the claim is that the number of free
ones is $I_q(k,m)$.

\subsection{A worked instance}\label{sec:worked}

Take $q=2$, $k=3$, and the content ``two $1$s and one $0$''. The group has
$\binom{3}{1,2}=3$ entries, $X_{011},X_{101},X_{110}$, and there is exactly one
Lyndon word of length $3$ with two $1$s, namely $011$. So the claim is that these
three entries have exactly one degree of freedom.

Here are the two equations. Both are instances of Lemma~\ref{lem:product}:
\begin{align}
  X_1X_{01}&=2X_{011}+X_{101}+X_{01},\label{eq:w1}\\
  X_1^2X_0&=2X_{011}+2X_{101}+2X_{110}+X_{01}+X_{10}.\label{eq:w2}
\end{align}
Everything of length $\le2$ is known, so moving it to the right,
\[
  2X_{011}+X_{101}=\underbrace{X_1X_{01}-X_{01}}_{\text{known}},\qquad
  2X_{011}+2X_{101}+2X_{110}=\underbrace{X_1^2X_0-X_{01}-X_{10}}_{\text{known}} .
\]
Two linear equations, three unknowns, and the coefficient matrix
$\left(\begin{smallmatrix}2&1&0\\2&2&2\end{smallmatrix}\right)$ has rank $2$. So
one degree of freedom remains, matching the count of one Lyndon word.

\subsection{The derived equations from the product rule}

Each equation above was produced by multiplying deck entries of short words:
\eqref{eq:w1} from $X_1\cdot X_{01}$, and \eqref{eq:w2} from $X_1\cdot X_1\cdot X_0$.
In both cases the short words are Lyndon---$1$, $01$, $0$---their lengths add up
to $k=3$ and their contents add up to $m$. That is the general recipe.

\begin{proposition}\label{prop:dof}
Fix $n$ and a content $m$ with $|m|=k\ge2$. Suppose the values $X_\gamma$ are
known for every word $\gamma$ with $|\gamma|\le k-1$. Then the vector of unknowns
$\left(X_w\right)_{c(w)=m}$ is confined to an affine subspace of dimension
exactly $I_q(k,m)$.
\end{proposition}

\begin{proof}
\emph{Step 1: each short-factor basis element gives one equation.} Let
$P=\ell_1\sh\cdots\sh\ell_r$ be one of the basis elements of
Lemma~\ref{lem:count-basis} whose factors all have length $<k$, and write
$P=\sum_w P_w\,w$. Multiplying the corresponding deck entries and applying
Lemma~\ref{lem:leading},
\[
  \prod_{i=1}^{r}X_{\ell_i}
  =\underbrace{\sum_{|w|=k}P_wX_w}_{\text{length }k}
  +\underbrace{\sum_{|\gamma|<k}c_\gamma X_\gamma}_{\text{shorter}},\qquad c_\gamma\in\ZZ .
\]
Every factor $X_{\ell_i}$ on the left is known, because $|\ell_i|<k$; every term
in the second sum is known, because it has length $<k$. Rearranging,
$\sum_{|w|=k}P_wX_w$ equals a known number.

\emph{Step 2: the equation involves only the unknowns of content $m$.} A shuffle
adds the contents of its factors, so $P$ is homogeneous of content $m$; that is,
$P_w=0$ unless $c(w)=m$. We stress that this uses Lemma~\ref{lem:leading}, which
speaks about length only. It would be wrong to argue instead that the shorter
terms $X_\gamma$ have content different from $m$ and may be ignored on that
ground: they need not (see Remark~\ref{rem:caution}). They are ignorable because
they are known, whatever their content.

\emph{Step 3: count the equations.} By Lemma~\ref{lem:count-basis} the number of
basis elements with all factors of length $<k$ is $\binom{k}{m}-I_q(k,m)$, so
Step 1 produces that many equations.

\emph{Step 4: the equations are independent.} The coefficient vector of the
equation coming from $P$ is $(P_w)_w$, i.e.\ $P$ itself expressed in the basis of
words. These vectors are distinct members of a basis of the space spanned by the
words of content $m$, hence linearly independent.

We therefore have $\binom{k}{m}-I_q(k,m)$ independent linear equations in
$\binom{k}{m}$ unknowns, so the solution set is an affine subspace of dimension
$I_q(k,m)$.
\end{proof}

\begin{remark}
The proof shows a little more than the statement: it exhibits the equations
explicitly, one for each way of writing $m$ as a sum of contents of Lyndon words
all shorter than $k$. In the worked instance those two ways were $\{1,01\}$ and
$\{1,1,0\}$, and the remaining way, $\{011\}$, uses a factor of length $3$ and
therefore yields no equation---which is precisely why one degree of freedom
survives.
\end{remark}

\begin{remark}\label{rem:exactness}
Two scopes should be distinguished. The dimension is exact for the system of
relations generated by Lemma~\ref{lem:product}. Whether the achievable decks fill
that affine subspace, or are confined by further relations of higher degree, is
not settled here; it is exactly the content of Conjecture~\ref{conj:main}.
\end{remark}

\subsection{Counting Lyndon words}

Proposition~\ref{prop:dof} says the number of free deck entries of length $k$ and
content $m$ equals $I_q(k,m)$. To use this we need a formula. Lyndon words are
the representatives of aperiodic rotation classes, and such classes are counted
by inclusion--exclusion on periods.

Take $q=2$, $k=4$, and content ``two $0$s and two $1$s''. There are
$\binom{4}{2,2}=6$ such words, in two rotation classes:
\[
\text{class I }\{0011,0110,1100,1001\}\ \text{(size 4)},\qquad
\text{class II }\{0101,1010\}\ \text{(size 2)} .
\]
Class I has the full four rotations; its members are aperiodic. Class II has only
two, because $0101=(01)^2$. A Lyndon word is strictly smaller than all its
rotations, which forces all its rotations to be distinct---so a Lyndon word is
aperiodic, and each aperiodic class contains exactly one. Hence
$I_2(4,(2,2))=(6-2)/4=1$, the class-I representative $0011$.

\begin{lemma}[Aperiodic classes have $k$ members, one Lyndon]\label{lem:aper}
Let $A_q(k,m)$ be the number of aperiodic words of length $k$ and content $m$.
Then $I_q(k,m)=A_q(k,m)/k$.
\end{lemma}

\begin{proof}
Rotation partitions the aperiodic words of length $k$ and content $m$ into
classes; rotation preserves content, so the classes stay inside our set. An
aperiodic word has $k$ distinct rotations, so every class has exactly $k$
members. Each class has a unique lexicographically least member, and that member
is smaller than all its rotations, hence Lyndon; conversely a Lyndon word is the
least member of its own class.
\end{proof}

\begin{lemma}[Every word is a repeated aperiodic word]\label{lem:repeat}
$\displaystyle\binom{k}{m}=\sum_{j\mid\gcd(m)}A_q\!\left(\frac kj,\frac mj\right).$
\end{lemma}

\begin{proof}
Let $w$ have length $k$ and content $m$, and let $d$ be its smallest period, so
$w=u^{j}$ with $u$ aperiodic of length $d$ and $j=k/d$. This factorization is
unique. Since $w$ consists of $j$ copies of $u$, the content of $u$ is $m/j$; in
particular $j$ divides every $m_i$, that is $j\mid\gcd(m)$. Conversely, for each
such $j$ every aperiodic $u$ of length $k/j$ and content $m/j$ yields the word
$u^j$, distinct $u$ giving distinct words.
\end{proof}

\begin{corollary}\label{cor:closed}
For $|m|=k$,
\[
  I_q(k,m)=\frac1k\sum_{j\mid\gcd(m)}\mu(j)\binom{k/j}{m/j},
  \qquad
  L_q(k)=\sum_{|m|=k}I_q(k,m)=\frac1k\sum_{j\mid k}\mu(j)\,q^{k/j}.
\]
\end{corollary}

\begin{proof}
Lemma~\ref{lem:repeat} is exactly the hypothesis of M\"obius inversion over the
divisor lattice of $\gcd(m)$; inverting and dividing by $k$ as in
Lemma~\ref{lem:aper} gives the first formula. For the second, sum over all $m$
with $|m|=k$ and interchange the two sums: a divisor $j$ contributes $\mu(j)$
times the number of words of length $k/j$ over the alphabet, that is
$\mu(j)q^{k/j}$; and $j$ ranges over the divisors of $k$.
\end{proof}

\begin{example}
$\gcd(2,2)=2$, so
$I_2(4,(2,2))=\tfrac14\bigl[\binom{4}{2,2}-\binom{2}{1,1}\bigr]=\tfrac{6-2}{4}=1$.
The term $-2$ is exactly the two periodic words $0101$ and $1010$.
\end{example}

\begin{remark}[When the M\"obius correction bites]\label{rem:whenbites}
When $\gcd(m)=1$ the sum has a single term and $I_q(k,m)=\frac1k\binom{k}{m}$: no
word of that content can be periodic, because a period would have to divide every
$m_i$. Over the binary alphabet the contents with $\gcd(m)>1$ are the constant
ones $(k,0)$ and $(0,k)$, for every $k\ge2$, together with those in which both
entries share a factor. The correction is therefore not a phenomenon of large
$k$: already at $k=2$ the content $(2,0)$ has $\gcd=2$ and the leading term alone
gives $\binom{2}{2,0}/2=1/2$, whereas $I_2(2,(2,0))=0$, correctly excluding the
non-Lyndon word $00$. What is special about $(2,2)$ at $k=4$ is that it is the
first binary content with $\gcd(m)>1$ for which $I_q(k,m)$ is nonzero, so that
the correction changes a positive count rather than reducing it to zero.
\end{remark}

\begin{remark}
The division by $k$ always yields an integer, since $A_q(k,m)$ counts a set
partitioned into classes of size $k$. This is a useful check: $\frac14\binom{4}{2,2}=\frac64$
is not an integer, which by itself signals that the M\"obius correction cannot be
omitted.
\end{remark}

\section{The Upper Bound}\label{sec:upper}

\subsection{The counting principle}

A $k$-deck is a list of numbers $X_u$, one for each word $u$ of length at most
$k$. To bound how many different decks can occur we ask how much information it
takes to write one down. Specifying the deck level by level,
Proposition~\ref{prop:dof} says that at level $j$ we need not give all $q^j$
entries: once levels $1,\dots,j-1$ are on the table, each content class at level
$j$ is pinned to an affine subspace, and we need only name a point of it. The
cost of level $j$ is
\[
  (\text{number of free coordinates})\times(\text{size of the range of each}),
\]
the first factor being $L_q(j)$ by Proposition~\ref{prop:dof} summed over content
classes, and the second being $O(n^j)$, since $0\le X_u\le\binom nj$ for
$|u|=j$.

\subsection{Level 1, and the $-1$}

At level $1$ the deck entries are the letter counts $X_{a_1},\dots,X_{a_q}$.
Proposition~\ref{prop:dof} does not apply ($k\ge2$ is assumed) and would
overcount: it would report $L_q(1)=q$ free coordinates, whereas the truth is
$q-1$, because
\[
  X_{a_1}+X_{a_2}+\cdots+X_{a_q}=n .
\]
The reason the proposition cannot see this relation is worth stating: every
equation it produces lives inside a single content class, whereas this one ties
all $q$ classes together. It is not a deck identity at all---it merely records
that we are looking at words of one fixed length. Being a single linear equation,
it removes exactly one dimension, whatever $q$ is. For $j\ge2$ no such correction
arises, because by then all the $X_a$ are known. Writing $q=1\cdot L_q(1)$, the
level-1 cost is $n^{1\cdot L_q(1)-1}$, which is the entire origin of the $-1$ in
$E_q(k)$.

\subsection{A worked example}

Let $q=2$, $k=3$:
\begin{center}
\begin{tabular}{clccc}
\toprule
level $j$ & free coordinates & how many & range of each & exponent\\
\midrule
1 & $w=X_1$ (as $X_0+X_1=n$) & 1 & $O(n)$   & 1\\
2 & $t=X_{01}$               & 1 & $O(n^2)$ & 2\\
3 & $X_{001},X_{011}$        & 2 & $O(n^3)$ & 6\\
\midrule
  &                          &   &          & \textbf{9}\\
\bottomrule
\end{tabular}
\end{center}
So $D_{2,3}(n)=O(n^9)$, reproducing the pattern $1,1,2$ found by hand in
Section~\ref{sec:idea}. For the DNA alphabet $L_4(1),L_4(2),L_4(3)=4,6,20$, so the
exponents are $3,12,60$ and $E_4(3)=75$.

\begin{theorem}\label{thm:upper}
For every fixed $q\ge2$ and $k\ge2$,
\[
  D_{q,k}(n)=O\!\left(n^{E_q(k)}\right),\qquad E_q(k)=\sum_{j=1}^{k}j\,L_q(j)-1 .
\]
In particular, for the binary alphabet ($q=2$), since $jL_2(j)$ is the number of
aperiodic binary words of length $j$ and hence at most $2^j$, we have
$E_2(k)\le\sum_{j=1}^{k}2^j-1=2^{k+1}-3$, so
\[
  D_{2,k}(n)=O\!\left(n^{4\cdot 2^{k-1}}\right) .
\]
\end{theorem}

\begin{proof}
We bound the number of tuples $(X_u)_{|u|\le k}$ by choosing them level by level.

\emph{Level 1.} The entries $(X_a)_{a\in\mathcal A}$ form a composition of $n$
into $q$ non-negative parts; there are $\binom{n+q-1}{q-1}=O(n^{q-1})$ of these.

\emph{Level $j$, $2\le j\le k$.} Suppose all lower levels are fixed and let $m$
be a content with $|m|=j$. By Proposition~\ref{prop:dof} the entries of that
content lie on an affine subspace $V_m$ of dimension $d=I_q(j,m)$. Choose $d$ of
the coordinates so that the projection of $V_m$ onto them is injective; this is
possible because $V_m$ has dimension $d$. A point of $V_m$ is then determined by
those $d$ numbers, each an integer in $\bigl[0,\binom nj\bigr]$, so the block
admits at most $\bigl(\binom nj+1\bigr)^{d}=O\!\left(n^{jI_q(j,m)}\right)$
values. Multiplying over the content classes of size $j$, whose dimensions sum to
$L_q(j)$, level $j$ offers $O\!\left(n^{jL_q(j)}\right)$ choices.

\emph{Multiplying.} Combining the levels gives
$O\!\left(n^{(q-1)+\sum_{j=2}^{k}jL_q(j)}\right)$, and since $L_q(1)=q$ the
level-1 exponent $q-1$ equals $1\cdot L_q(1)-1$.
\end{proof}

\section{A Matching Lower Bound for \texorpdfstring{$k=2$}{k=2}}\label{sec:k2}

This section is independent of Sections~\ref{sec:product}--\ref{sec:upper} except
for the value of $E_q(2)$, and is elementary throughout.

\begin{remark}[The case $k=2$, by hand]\label{rem:k2byhand}
Fix the letter counts $n_1,\dots,n_q$---that is $O(n^{q-1})$ choices. Then $X_{aa}$
is determined, since $X_{aa}=\binom{n_a}{2}$; and for $a\ne b$ the two entries
$X_{ab},X_{ba}$ satisfy $X_{ab}+X_{ba}=n_an_b$, so each unordered pair carries a
single free entry ranging over $O(n^2)$ values. Hence
\[
  E_q(2)=(q-1)+2\binom q2=q^2-1,
\]
in agreement with $\sum_{j\le2}jL_q(j)-1=q+2\binom q2-1$ for every $q$.
\end{remark}

\begin{theorem}\label{thm:k2}
For every $q\ge2$, $\ D_{q,2}(n)=\Theta\!\left(n^{q^2-1}\right)$.
\end{theorem}

\begin{proof}
The upper bound is Theorem~\ref{thm:upper} with $E_q(2)=q^2-1$. For the lower
bound, fix letter counts $n_1,\dots,n_q$ with $\sum_a n_a=n$ and
$n_a\ge n/(2q)$ for every $a$; there are $\Theta(n^{q-1})$ such choices.

Put $m_a:=\lfloor n_a/(q-1)\rfloor$ and lay the word out as a concatenation of
$\binom q2$ consecutive \emph{segments} $S_{ab}$, one for each unordered pair
$a<b$, taken in any fixed order, followed by a final \emph{remainder} block. The
segment $S_{ab}$ contains $m_a$ copies of $a$ and $m_b$ copies of $b$ and nothing
else; each letter $a$ thus occurs in the $q-1$ segments involving $a$, using
$(q-1)m_a\le n_a$ copies, and the unused copies are placed in the remainder block
in sorted order.

Fix $a<b$ and count the occurrences of $ab$ as a subsequence. They are of three
kinds. Occurrences with both letters inside one segment can only arise inside
$S_{ab}$, since no other segment contains both $a$ and $b$; occurrences with both
letters inside the remainder block are determined, since that block is a fixed
sorted arrangement; and occurrences with the two letters in different blocks are
determined by the block order together with the number of $a$'s and $b$'s in each
block, all of which are fixed. Hence
\[
  X_{ab}=c_{ab}+z_{ab},
\]
where $c_{ab}$ depends only on the data already fixed, and $z_{ab}$ is the number
of occurrences of $ab$ inside $S_{ab}$.

Within $S_{ab}$ we may choose any arrangement of $m_a$ letters $a$ and $m_b$
letters $b$. As the arrangement runs from $b^{m_b}a^{m_a}$ to $a^{m_a}b^{m_b}$,
transposing one adjacent unequal pair at a time, $z_{ab}$ increases by exactly one
at each step, so $z_{ab}$ realises every value in $[0,m_am_b]$. The choices in
different segments are independent, since $z_{ab}$ depends only on the
arrangement inside $S_{ab}$.

Therefore, at fixed letter counts, the vector $(X_{ab})_{a<b}$ takes at least
$\prod_{a<b}(m_am_b+1)=\Theta\!\left(n^{2\binom q2}\right)$ distinct values.
Distinct such vectors give distinct $2$-decks, because $X_{aa}$ and $X_{ba}$ are
determined by the letter counts and $X_{ab}$. Multiplying by the
$\Theta(n^{q-1})$ choices of letter counts gives
$\Omega\!\left(n^{(q-1)+q^2-q}\right)=\Omega\!\left(n^{q^2-1}\right)$.
\end{proof}

\section{A Matching Lower Bound for \texorpdfstring{$q=2$, $k=3$}{q=2, k=3}}\label{sec:lower3}

This section is self-contained and uses nothing from
Sections~\ref{sec:product}--\ref{sec:upper}.

\subsection{Overview}

We already know, from Theorem~\ref{thm:upper}, that $D_{2,3}(n)=O(n^9)$: no
binary word's $3$-deck needs more than nine ``digits'' of information to
pin down. The question in this section is whether that many digits are
\emph{actually necessary} --- that is, whether there really are
$\Omega(n^9)$ distinguishable $3$-decks, or whether the true count is
smaller. We show the bound is tight: $D_{2,3}(n)=\Theta(n^9)$.

The strategy is very concrete. We are going to \emph{build}, explicitly,
a large collection of binary words that are guaranteed to have pairwise
distinct $3$-decks, and then just count them. The construction has two
ingredients.

\emph{Ingredient 1: a compression.} A binary word's entire $3$-deck (all
eight numbers $X_{000},X_{001},\dots,X_{111}$) turns out to be completely
determined by only four numbers: the length $w$ of ones, and three
subsequence counts $t$, $A$, $B$ defined below. So instead of comparing
$8$-tuples, we only need to produce many binary words whose
$(w,t,A,B)$-quadruples differ. This is Lemma~\ref{lem:decode}.

\emph{Ingredient 2: independent dials.} We then design a family of binary
words, laid out as four separate ``zones,'' each zone acting as a dial
that changes exactly one of $t$, $A$, $B$, while leaving the other two
untouched. If a dial for $t$ can be turned $\Theta(n^2)$ ways, a dial for
$A$ can be turned $\Theta(n^3)$ ways, and a dial for $B$ can be turned
$\Theta(n^3)$ ways, and the three dials do not interfere with each other,
then simply multiplying gives $\Theta(n)\cdot\Theta(n^2)\cdot\Theta(n^3)\cdot\Theta(n^3)=\Theta(n^9)$
distinguishable words. Building the dial for $t$ and the dial for $A$ is
easy; almost the entire technical difficulty of this section is in
building the dial for $B$, because --- as we explain in
Remark~\ref{rem:barrier} --- a ``dial'' for $B$ cannot be a single simple
move the way the dials for $t$ and $A$ are. It has to be a whole
\emph{chain} of small moves.

\subsection{A partition encoding}

To make ``zones'' precise we first translate a binary word into a picture.

Let $X\in\{0,1\}^n$ have $w$ ones and $n_0=n-w$ zeros. For $i=1,\dots,n_0$
let $\mu_i$ be the number of ones occurring \emph{after} the $i$-th zero.
For example, in the word $X=0\,0\,1\,1\,0\,1$ the zeros sit at positions
$1,2,5$, and the numbers of ones following them are $\mu_1=3$, $\mu_2=3$,
$\mu_3=1$. Reading the zeros from left to right can only ever leave the
same or fewer ones ahead of you, so
\[
  \mu_1\ge\mu_2\ge\cdots\ge\mu_{n_0}\ge0 .
\]
That is, $\mu=(\mu_1,\dots,\mu_{n_0})$ is a \emph{partition}: a
non-increasing staircase of integers, each one at most $w$ (there are only
$w$ ones in total). Conversely, any such staircase inside the
$n_0\times w$ box comes from exactly one word. So words with $w$ ones
correspond exactly, one-to-one, to staircases $\mu$ fitting in that box,
and everything below is phrased in terms of the staircase rather than the
word itself --- this is what lets us talk about ``moving one unit from
row $i$ to row $j$'' as a legitimate, checkable operation.

With this picture, define
\[
  t:=X_{01}=\sum_i\mu_i,\qquad A:=X_{001}=\sum_i(i-1)\mu_i,\qquad
  B:=X_{011}=\sum_i\binom{\mu_i}{2},
\]
and write $Q:=\sum_i\mu_i^2$, so that $B=\tfrac12(Q-t)$, and
$d_i:=\mu_i-\mu_{i+1}$ for the gap between consecutive rows. In words:
$t$ just adds up the staircase heights; $A$ weights each row by its
position before adding; $B$ counts, in each row, how many pairs of ones
sit there (which is why it involves $\binom{\mu_i}{2}$, a quantity that
grows quadratically rather than linearly in $\mu_i$). It is this
quadratic character of $B$ that will make it the hard one to control.

\begin{lemma}\label{lem:decode}
With the notation above, $X_{000}=\binom{n_0}{3}$, $X_{111}=\binom w3$, and
\begin{align*}
  X_{010}&=(n_0-1)t-2A, & X_{100}&=\tbinom{n_0}{2}w-A-X_{010},\\
  X_{101}&=(w-1)t-2B,   & X_{110}&=\tbinom w2 n_0-B-X_{101}.
\end{align*}
In particular, distinct tuples $(w,t,A,B)$ give distinct $3$-decks.
\end{lemma}

\begin{proof}
Let $\nu_j$ be the number of zeros preceding the $j$-th one. Counting $(0,0,1)$
patterns by their one gives $A=\sum_j\binom{\nu_j}{2}$, hence
$\sum_j\nu_j^2=2A+t$; counting $(0,1,0)$ patterns by their one gives
$X_{010}=\sum_j\nu_j(n_0-\nu_j)=n_0t-(2A+t)$. Dually, $B=\sum_i\binom{\mu_i}{2}$
gives $\sum_i\mu_i^2=2B+t$, and counting $(1,0,1)$ patterns by their zero gives
$X_{101}=\sum_i(w-\mu_i)\mu_i=wt-(2B+t)$. The two remaining identities are the
partitions of $\binom{n_0}{2}w$ into $X_{001}+X_{010}+X_{100}$ and of
$\binom w2 n_0$ into $X_{011}+X_{101}+X_{110}$.
\end{proof}

By Lemma~\ref{lem:decode}, it suffices to exhibit $\Omega(n^9)$ achievable
tuples $(w,t,A,B)$; this is the only thing the rest of the section does.

\subsection{A move that fixes both \texorpdfstring{$t$}{t} and \texorpdfstring{$A$}{A}}

\textbf{Idea.} We want a small, local edit to the staircase $\mu$ that
changes $B$ but leaves $t$ and $A$ exactly as they were. The trick is to
add one unit at two rows and remove one unit at two other rows, arranged
so the additions and removals cancel out in both the plain sum ($t$) and
the position-weighted sum ($A$) --- but \emph{not} in $B$, since $B$ is
quadratic and does not cancel the same way. Concretely: pick two pairs of
rows at the same spacing $e$ apart, add to the first row of each pair,
subtract from the second.

\begin{lemma}[Four-row move]\label{lem:move}
Let $e\ge1$ and let $a,\,a+e,\,c,\,c+e$ be four \emph{distinct} indices in
$\{1,\dots,n_0\}$. Define $M(a,c;e)$ to be the perturbation
\[
  \mu_a\mathrel{+}=1,\qquad \mu_{a+e}\mathrel{-}=1,\qquad
  \mu_c\mathrel{-}=1,\qquad \mu_{c+e}\mathrel{+}=1 .
\]
Then $\Delta t=0$ and $\Delta A=0$, and, writing $G_x:=\mu_x-\mu_{x+e}$ for the
profile before the move, $\ \Delta B=G_a-G_c+2$.
\end{lemma}

\begin{proof}
Write $\delta$ for the perturbation vector. Then $\sum_i\delta_i=1-1-1+1=0$,
giving $\Delta t=0$, and
$\sum_i(i-1)\delta_i=(a-1)-(a+e-1)-(c-1)+(c+e-1)=-e+e=0$, giving $\Delta A=0$.
For the third quantity $\Delta Q=2\sum_i\mu_i\delta_i+\sum_i\delta_i^2$; since
the four indices are distinct, $\sum_i\delta_i^2=4$ and
$\sum_i\mu_i\delta_i=G_a-G_c$. Hence $\Delta Q=2(G_a-G_c)+4$, and
$\Delta B=\tfrac12\Delta Q$ because $\Delta t=0$.
\end{proof}

\begin{remark}
The hypothesis that the four indices be distinct is not cosmetic: if $a+e=c$ that
row receives $-2$, $\sum_i\delta_i^2=6$, and the formula changes. In the
construction below the four rows are pairwise at distance at least $L\ge2$, hence
pairwise non-adjacent, so the legality of $M(a,c;e)$ may be checked at each of
the four rows separately against unchanged neighbours.
\end{remark}

\begin{remark}[Why $B$ needs many small moves rather than one big one]\label{rem:barrier}
It is worth explaining, in plain terms, why $t$ and $A$ can each be swept
by one flexible family of moves (Lemmas~\ref{lem:tsweep} and
\ref{lem:asweep} below), while $B$ needs the elaborate chain built over
the rest of this section. The reason is that $t$ and $A$ are
\emph{linear} in $\mu$ --- shifting one unit of height from row $i$ to
row $j$ changes them by a fixed amount regardless of how much height is
already there --- whereas $B$ is quadratic, so how much it changes
depends on the current heights themselves.

To see the consequence, suppose the moves used to change $B$ are confined
to a single family: an arithmetic progression of gap $s$, split into $m$
sub-blocks, each contributing $\Delta B=2\sum_\kappa c_\kappa^2$ with
$c_\kappa\le s$ (this is the family used in \cite[Lem.~26]{chrisnata} and
in an earlier version of the present work). The whole family must fit
inside the $n_0\times w$ box, which forces $ms=O(w)$, and therefore
\[
  \Delta B\le 2ms^2=2s(ms)=O(sw)=O(w^2)=O(n^2)
\]
\emph{no matter how} $m$ and $s$ are chosen. More generally, any single
perturbation $\delta$ satisfies
$\bigl|\sum_i\mu_i\delta_i\bigr|\le w\sum_i|\delta_i|$, so a family that
only ever disturbs $O(n)$ cells can never move $B$ by more than $O(n^2)$:
reaching the full $\Theta(n^3)$ range of $B$ requires disturbing
$\Theta(n^2)$ cells, a constant fraction of the entire diagram. This is
precisely why earlier arguments (including an earlier version of this
paper) that varied $B$ using a single bounded family topped out at
$\Omega(n^8)$ overall rather than $\Omega(n^9)$: they were, without
realizing it, running into this ceiling. What follows is the fix: instead
of one family, we use a long \emph{chain} of the four-row moves of
Lemma~\ref{lem:move}, applied one after another, so that the total
number of cells disturbed accumulates to $\Theta(n^2)$ over the course of
the chain even though each individual move disturbs only four cells.
\end{remark}

\subsection{The configuration}\label{sec:config}

\textbf{Idea.} We now lay out one specific staircase $\mu^{\circ}$,
divided top to bottom into labelled zones, each zone dedicated to one job:
a zone $\mathcal F$ for fine-tuning $B$ by small amounts, a zone
$\mathcal C$ for the long $B$-chain that supplies the bulk of $B$'s range,
a zone $\mathcal T$ for sweeping $t$, and a zone $\mathcal A$ for sweeping
$A$. The staircase must keep decreasing from top to bottom throughout, so
between zones we insert small extra ``drops'' in height; these exist purely
to give each zone enough room to be worked on without interfering with its
neighbours, and have no other role.

Fix $n$ large and $w\in[0.68n,0.75n]$; put $n_0=n-w$. Set
\[
  s:=\lfloor 0.006\,n\rfloor,\qquad
  L:=\text{the odd integer nearest }0.015\,n,
\]
and split the remaining rows evenly,
$|\mathcal T|=\lfloor(n_0-32-4L)/2\rfloor$ and
$|\mathcal A|=n_0-32-4L-|\mathcal T|$. Define $\mu^{\circ}$ by running a value
$v$ downwards from $v=M:=w-s$ and recording it in the following blocks, in order:

\begin{center}
\begin{tabular}{llll}
\toprule
block & rows & per-row drop & terminal drop\\
\midrule
$\mathcal F$ & $32$ & $s$ & $L$\\
$P_1$ & $L$ & $1$ & ---\\
$P_2$ & $L$ & $1$ & ---\\
$P_3$ & $L$ & $3$ & $L$\\
$P_4$ & $L$ & $3$ & $L$\\
$\mathcal T$ & $|\mathcal T|$ & $1$ & ---\\
$\mathcal A$ & $|\mathcal A|$ & $1$ & ---\\
\bottomrule
\end{tabular}
\end{center}

Write $\mathcal C:=P_1\cup P_2\cup P_3\cup P_4$. The block $\mathcal F$ carries a
fine $B$-sweep, $\mathcal C$ a coarse $B$-chain, $\mathcal T$ the $t$-sweep and
$\mathcal A$ the $A$-sweep. The terminal drops of size $L$ after $\mathcal F$,
at the $P_3/P_4$ junction, and after $P_4$ are what allow those junctions to
survive the whole chain; without them the construction fails after one sweep.

\begin{lemma}[Feasibility]\label{lem:feas}
For $n$ sufficiently large, $\mu^{\circ}$ is a partition inside the
$n_0\times w$ box.
\end{lemma}

\begin{proof}
Each block is non-increasing and the terminal drops are non-negative, so
$\mu^{\circ}$ is non-increasing, and $\mu^{\circ}_1=M=w-s\le w$. It remains to
check $\mu^{\circ}_{n_0}\ge0$, i.e.\ that the total drop
$32s+11L+(n_0-32-4L)$ does not exceed $M$; for the stated parameters and
$w\ge0.68n$ this is at most $0.62\,n\le M$.
\end{proof}

\subsection{The four sweeps}

We now turn each zone's dial in turn, checking that turning it does not
disturb the other zones.

\textbf{Idea (the $t$-sweep).} Inside $\mathcal T$, raising a single row's
height by one --- as long as the row below still fits under the row above
--- changes $t$ by exactly $+1$ and changes nothing else. Repeating this
across the whole zone, from its starting shape up to completely flat,
sweeps $t$ through a whole range of $\Theta(n^2)$ consecutive integers.

\begin{lemma}[$t$-sweep]\label{lem:tsweep}
Raising a single row $j$ of $\mathcal T$ other than its first by one is legal
whenever $\mu_j+1\le\mu_{j-1}$, and has $\Delta t=+1$. Raising the rows of
$\mathcal T$ until the block is flat at the value of its first row realises
$\binom{|\mathcal T|}{2}+1=\Theta(n^2)$ consecutive values of $t$. All rows
outside $\mathcal T$ are pointwise unchanged, and after flattening the last row
of $\mathcal T$ exceeds the first row of $\mathcal A$ by exactly $|\mathcal T|$.
\end{lemma}

\textbf{Idea (the $A$-sweep).} Inside $\mathcal A$, moving one unit of height
from row $i$ to the row just below it changes $A$ by exactly $+1$ (since
$A$ weights row $i$ by $i-1$ and row $i+1$ by $i$) and leaves $t$
unchanged (the total height moved is the same). Starting from a
``top-heavy'' shape and repeatedly pushing height downward until the
block becomes a perfect staircase realises $\Theta(n^3)$ consecutive
values of $A$.

\begin{lemma}[$A$-sweep]\label{lem:asweep}
For $i,i+1\in\mathcal A$ the operation $\mu_i\mapsto\mu_i-1$,
$\mu_{i+1}\mapsto\mu_{i+1}+1$ is legal iff $\mu_i\ge\mu_{i+1}+2$, and has
$\Delta t=0$, $\Delta A=+1$; it alters only rows $i$ and $i+1$. Sweeping from the
top-heavy configuration of $\mathcal A$ to the staircase realises
$\Theta(|\mathcal A|^3)=\Theta(n^3)$ consecutive values of $A$ at fixed $t$. All
rows outside $\mathcal A$ are pointwise unchanged.
\end{lemma}

\begin{proof}
Admissibility requires $\mu_i-1\ge\mu_{i+1}+1$. The number of cells is unchanged,
so $\Delta t=0$, and moving one cell from row $i$ to row $i+1$ changes
$A=\sum_i(i-1)\mu_i$ by one. Index the rows of $\mathcal A$ by
$p=0,\dots,|\mathcal A|-1$; the block carries a fixed number of cells and each row
is capped by the row above it. Restricted to $\mathcal A$, the quantity $A$ equals
a constant plus $\sum_p p\mu_p$. At the top-heavy configuration, cells packed at
the cap in the first $\approx|\mathcal A|/2$ rows, $\sum_p p\mu_p\approx|\mathcal A|^3/8$;
the process halts only when all adjacent gaps are at most one, that is at the
staircase, where $\sum_p p\mu_p\approx|\mathcal A|^3/6$. The run of consecutive
values therefore has length $\left(\frac16-\frac18\right)|\mathcal A|^3+O(|\mathcal A|^2)$.
\end{proof}

\textbf{Idea (the fine $B$-sweep).} Inside $\mathcal F$, which is laid out
as four small identical sub-blocks of $8$ rows each, we apply a
symmetric $\pm c_\kappa$ perturbation within each sub-block. Because the
rows involved are symmetric around the middle of the sub-block, the plain
sum and the weighted sum both cancel exactly ($\Delta t=\Delta A=0$), while
the quadratic quantity $B$ picks up $2c_\kappa^2$ per sub-block. Since
every non-negative integer up to $s^2$ can be written as a sum of (at
most) four squares, choosing the four $c_\kappa$'s appropriately lets $B$
hit every value in a range of size $\Theta(s^2)=\Theta(n^2)$ --- this is
the ``fine'' adjustment that fills in the gaps left by the much coarser
chain of Section~\ref{sec:coarse-explain}.

\begin{lemma}[Fine $B$-sweep]\label{lem:fine}
Let $c_1,\dots,c_4\in[0,s]$ and apply, inside $\mathcal F$ and for each
$\kappa\in\{1,2,3,4\}$, the perturbation $+c_\kappa$ at rows $8\kappa-7$ and
$8\kappa-1$ and $-c_\kappa$ at rows $8\kappa-5$ and $8\kappa-3$. This is legal,
has $\Delta t=\Delta A=0$ and $\Delta B=2\sum_\kappa c_\kappa^2$; consequently $B$
attains every value $B+2m$ with $0\le m\le s^2$. All rows outside $\mathcal F$ are
pointwise unchanged, as is row $32$.
\end{lemma}

\begin{proof}
The signs sum to zero and the index identity
$(8\kappa-7)+(8\kappa-1)-(8\kappa-5)-(8\kappa-3)=0$ gives $\Delta t=\Delta A=0$.
Since $\mu^{\circ}$ is affine in $i$ on $\mathcal F$, the same identity gives
$\sum_i\mu_i\delta_i=0$, whence $\Delta Q=\sum_i\delta_i^2=4\sum_\kappa c_\kappa^2$
and $\Delta B=2\sum_\kappa c_\kappa^2$. The perturbed rows inside a sub-block are
pairwise non-adjacent and row $8\kappa$ is untouched, so every perturbed row has
both neighbours untouched; the binding constraint is a gap of exactly $s$ in each
case, so all $c\in[0,s]^4$ are simultaneously admissible, and row $1$ satisfies
$\mu_1+c_1\le M+s=w$. By Lagrange's four-square theorem every $m\in[0,s^2]$ is a
sum of four squares each at most $m\le s^2$.
\end{proof}

\subsection{The coarse chain}\label{sec:coarse-explain}

\textbf{Idea.} The fine sweep above only reaches a range of size
$\Theta(n^2)$; to reach the full $\Theta(n^3)$ that the exponent count
demands, we chain together many four-row moves from Lemma~\ref{lem:move}
inside $\mathcal C=P_1\cup P_2\cup P_3\cup P_4$. Each single move in the
chain uses one row from $P_1$ paired with the row directly below it in
$P_2$, and one row from $P_3$ paired with the one below it in $P_4$; by
Lemma~\ref{lem:move} this changes $B$ by a fixed, computable amount while
leaving $t$ and $A$ exactly unchanged. Running this move down every row
of the four blocks is one \emph{sweep}; repeating the sweep many times
(each time the four blocks having shifted slightly relative to each
other) accumulates $\Theta(n^2)$ moves, each changing $B$ by $\Theta(n)$,
for a total range of $\Theta(n^3)$. The bulk of the proof below is simply
bookkeeping: checking that after each sweep the staircase shape is still
legal (no row overtakes the row above it) and computing exactly how much
$B$ moves at each step.

Index the rows of $P_r$ as $P_r[0],\dots,P_r[L-1]$. For $0\le\mathrm{idx}\le L-1$
let
\[
  M_{\mathrm{idx}}:=M\bigl(P_1[\mathrm{idx}],P_3[\mathrm{idx}];e=L\bigr),
\]
so that the four affected rows are $P_1[\mathrm{idx}],\dots,P_4[\mathrm{idx}]$.
\emph{Sweep $k$} consists of $M_0,M_1,\dots,M_{L-1}$ applied in this order. Put
$K:=\lfloor L/2\rfloor-1$.

\begin{lemma}[Coarse chain]\label{lem:coarse}
Sweeps $0,1,\dots,K-1$ are legal; throughout, $t$ and $A$ are unchanged and all
rows outside $\mathcal C$ are pointwise unchanged. The move $M_{\mathrm{idx}}$ of
sweep $k$ satisfies, independently of $\mathrm{idx}$,
\[
  \Delta B=-\bigl(3L-4k-2\bigr)<0,
\]
so the chain consists of $LK$ strictly decreasing steps, each of size at most
$3L-2$, with total descent
\[
  B_{\mathrm{start}}-B_{\mathrm{end}}=\sum_{k=0}^{K-1}L\,(3L-4k-2)=L\,K\,(3L-2K)=\Theta(n^3).
\]
\end{lemma}

\begin{proof}
\emph{Bookkeeping.} An induction shows that at the start of sweep $k$, relative to
$\mu^{\circ}$, every row of $P_1$ and $P_4$ has been raised by exactly $k$ and
every row of $P_2$ and $P_3$ lowered by exactly $k$; and that after the first
$\mathrm{idx}$ moves of sweep $k$ this holds with $k+1$ in place of $k$ for the
rows of smaller index.

\emph{The increment.} In $\mu^{\circ}$ the drop from $P_1[\mathrm{idx}]$ to
$P_2[\mathrm{idx}]$ runs over $L$ gaps each equal to $1$, so it is $L$; the drop
from $P_3[\mathrm{idx}]$ to $P_4[\mathrm{idx}]$ runs over $L$ gaps of which $L-1$
equal $3$ and one is the junction gap $3+L$, so it is $3(L-1)+(3+L)=4L$. By the
bookkeeping, when $M_{\mathrm{idx}}$ is applied
$G_{P_1[\mathrm{idx}]}=L+2k$ and $G_{P_3[\mathrm{idx}]}=4L-2k$, and
Lemma~\ref{lem:move} gives $\Delta B=(L+2k)-(4L-2k)+2=-(3L-4k-2)$. For
$k\le K-1\le L/2-2$ we have $4k+2\le2L-6<3L$, so every step is strictly negative
and of size at most $3L-2$. Summing,
$\sum_{k=0}^{K-1}L(3L-4k-2)=L\bigl[(3L-2)K-2K(K-1)\bigr]=LK(3L-2K)$, which is
$\bigl(\tfrac12+o(1)\bigr)L^3$ for $K=\lfloor L/2\rfloor-1$.

\emph{Legality.} The four affected rows are pairwise at distance at least
$L\ge2$, so it suffices to check each against its unchanged neighbours. Writing
$k$ for the number of completed sweeps:
(i) raising $P_1[\mathrm{idx}]$ needs the gap above it to be $\ge1$; for
$\mathrm{idx}\ge1$ that gap is $1+(k+1)-k=2$, and for $\mathrm{idx}=0$ the row
above is the last row of $\mathcal F$ and the gap is $s+L-k\ge s+L/2$.
(ii) lowering $P_2[\mathrm{idx}]$ needs the gap below it to be $\ge1$; the row
below has been lowered by the same amount, so that gap is still $1$.
(iii) lowering $P_3[\mathrm{idx}]$ needs the gap below it to be $\ge1$; for
$\mathrm{idx}\le L-2$ it is $3$, and for $\mathrm{idx}=L-1$ the row below is
$P_4[0]$, giving $3+L-k-(k+1)=L+2-2k\ge2$.
(iv) raising $P_4[\mathrm{idx}]$ needs the gap above it to be $\ge1$; for
$\mathrm{idx}\ge1$ it is $3-1=2$, and for $\mathrm{idx}=0$ the row above is
$P_3[L-1]$ and the gap is $3+L-2k\ge3$.
(v) the junction below $P_4$ has gap $3+L-(k+1)\ge3$, and $\mu_1$ is untouched.
All bounds use only $k\le K\le L/2$.
\end{proof}

\textbf{Idea (putting it together).} We now have four dials --- $\mathcal T$,
$\mathcal A$, $\mathcal F$, $\mathcal C$ --- and each has been shown to
work in isolation. The next lemma is the check that they really are
independent: turning one dial never resets or interferes with another,
because they live on disjoint rows and only ever touch each other at a
shared boundary, and those boundaries were built with enough slack (the
terminal drops of Section~\ref{sec:config}) to absorb whichever dial
happens to be turned first.

\begin{lemma}[Independence]\label{lem:indep}
The four families of Lemmas~\ref{lem:tsweep}, \ref{lem:asweep}, \ref{lem:fine}
and \ref{lem:coarse} are supported on the pairwise disjoint blocks $\mathcal T$,
$\mathcal A$, $\mathcal F$, $\mathcal C$. Consequently their increments compose
additively, and the legality bounds of Lemmas~\ref{lem:fine} and
\ref{lem:coarse} hold uniformly over every configuration of $\mathcal T$ and
$\mathcal A$ reachable by Lemmas~\ref{lem:tsweep} and~\ref{lem:asweep}.
\end{lemma}

\begin{proof}
For a perturbation $\delta$ one has $\Delta t=\sum_i\delta_i$,
$\Delta A=\sum_i(i-1)\delta_i$ and $\Delta Q=2\sum_i\mu_i\delta_i+\sum_i\delta_i^2$,
all of which involve $\mu$ only on $\operatorname{supp}\delta$; disjointness of
supports gives additivity. For uniformity, the only interaction between blocks is
at their junctions. The $t$-sweep only raises rows of $\mathcal T$ and never its
first row, so the $\mathcal C/\mathcal T$ junction gap can only grow; the
$A$-sweep never alters $\mathcal T$; and neither touches $\mathcal F$ or
$\mathcal C$. The bounds (i)--(v) above thus remain valid verbatim.
\end{proof}

\textbf{Idea (combining the coarse chain and the fine sweep).} The coarse
chain visits $\Theta(n^2)$ values of $B$, spaced up to $\Theta(n)$ apart,
alternating parity because $L$ was chosen odd. At each of these ``stepping
stones,'' the fine sweep of Lemma~\ref{lem:fine} fills in a
range of $\Theta(n^2)$ nearby integers, which is far more than the gap to
the next stepping stone. So every integer in between two consecutive
stepping stones gets covered by the fine sweep anchored at one end or the
other, and the whole interval from the chain's start to its end is
achieved with no gaps.

\begin{proposition}\label{prop:Brange}
Let $n$ be sufficiently large and fix $w\in[0.68n,0.75n]$ together with any values
of $t$ and $A$ realised by Lemmas~\ref{lem:tsweep} and~\ref{lem:asweep}. Then $B$
attains at least $L\,K\,(3L-2K)+1=\Theta(n^3)$ distinct values with $w$, $t$ and
$A$ held fixed.
\end{proposition}

\begin{proof}
Let $v_0>v_1>\cdots>v_{LK}$ be the values of $B$ at the nodes of the chain of
Lemma~\ref{lem:coarse}, so $v_0-v_{LK}=LK(3L-2K)$ and $v_j-v_{j+1}\le3L-2$. Since
$L$ is odd, each increment $3L-4k-2$ is odd, so consecutive nodes have opposite
parities. By Lemma~\ref{lem:fine} the fine sweep, applied at any node, adds every
even value in $[0,2s^2]$, and by Lemma~\ref{lem:indep} it may be applied at every
node without disturbing $t$, $A$ or the chain. For the stated parameters $2s^2$
exceeds $3L-2$ once $n$ is large enough, so for each $j$ the interval
$[v_{j+1},v_j]$ is covered: the fine sweep at $v_{j+1}$ supplies all integers of
one parity in it, and the fine sweep at the node below $v_j$ supplies the other.
Hence every integer of $[v_{LK},v_0]$ is attained.
\end{proof}

\subsection{The bound}

Putting the four dials together and counting how many settings each one
admits gives the theorem.

\begin{theorem}\label{thm:main9}
$D_{2,3}(n)=\Theta\!\left(n^{9}\right)$.
\end{theorem}

\begin{proof}
The upper bound $D_{2,3}(n)=O(n^{E_2(3)})=O(n^9)$ is Theorem~\ref{thm:upper}. For
the lower bound we count achievable tuples $(w,t,A,B)$; by Lemma~\ref{lem:decode}
distinct tuples give distinct $3$-decks. The count is nested, each stage
preserving what was fixed before it.
\begin{itemize}
\item Choose $w\in[0.68n,0.75n]$: $\Theta(n)$ values, by Lemma~\ref{lem:feas}.
\item Choose a configuration of $\mathcal T$: $\Theta(n^2)$ values of $t$, by
Lemma~\ref{lem:tsweep}. This also shifts $A$ and $B$, which is harmless because
distinct $t$ already gives distinct tuples.
\item At that $t$, choose a configuration of $\mathcal A$: $\Theta(n^3)$ values of
$A$ with $t$ preserved, by Lemma~\ref{lem:asweep}.
\item At that $(t,A)$, run the chain of Lemma~\ref{lem:coarse} together with the
fine sweep of Lemma~\ref{lem:fine}: $\Theta(n^3)$ values of $B$ with $t$ and $A$
preserved, by Proposition~\ref{prop:Brange}.
\end{itemize}
The last two stages survive the earlier choices by Lemma~\ref{lem:indep}. Hence
$D_{2,3}(n)\ge\Theta(n)\cdot\Theta(n^2)\cdot\Theta(n^3)\cdot\Theta(n^3)=\Omega(n^9)$.
\end{proof}

\begin{remark}
Chrisnata et al.\ \cite{chrisnata} obtain $\Omega(n^6)$ as
$\Theta(n)\cdot\Theta(n^2)\cdot\Theta(n^3)$ from $w$, $t$ and $X_{011}$, never
using $X_{001}$. In the language of Section~\ref{sec:lyndon}, $X_{001}$ and
$X_{011}$ are precisely the $L_2(3)=2$ free parameters at level three,
corresponding to the Lyndon words $001$ and $011$; the earlier bound varied one of
them over its full range, and the present one varies both. What makes the second
sweep possible is Lemma~\ref{lem:move}, whose two constraints $\Delta t=\Delta A=0$
hold identically rather than by cancellation against a fixed affine profile;
Remark~\ref{rem:barrier} explains why no bounded family of perturbations can
achieve the same.
\end{remark}

\section{The bound conjecture}\label{sec:numerics}

\begin{conjecture}\label{conj:main}
$D_{q,k}(n)=\Theta\!\left(n^{E_q(k)}\right)$ for every fixed $q\ge2$, $k\ge1$.
\end{conjecture}

The conjecture holds for $k=1$, where $D_{q,1}(n)=\binom{n+q-1}{q-1}$; for $k=2$
and every $q$, by Theorem~\ref{thm:k2}; and for $q=2$, $k=3$, by
Theorem~\ref{thm:main9}. A general matching lower bound still needs to be proved. 

\section{Conclusion}

We have analysed the free parameters in a $k$-deck once the shorter decks are
fixed, obtaining the count of Lyndon words of the relevant content, and deduced an
upper bound on the number of distinct $k$-decks valid for every alphabet size.
 We proved matching lower bounds for $k=2$ over every alphabet and for $q=2$, $k=3$.
The main open problem is Conjecture~\ref{conj:main} beyond these cases.

\end{document}